\documentclass[10pt,leqno]{amsart}
\usepackage{graphicx}
\usepackage{mathrsfs}

\usepackage{amssymb,amsthm,amsmath}
\usepackage{xcolor,paralist,hyperref,fancyhdr,etoolbox}
\newtheorem{theorem}{Theorem}[]

\newtheorem{example}[theorem]{Example}
\newtheorem{lemma}[theorem]{Lemma}

\newtheorem{remark}[theorem]{Remark}

\newcommand{\X}{\mathcal{X}}
\newcommand{\AC}{AC([0,T], \X)}

\newcommand{\vf}{\varphi}

\newcommand{\eproof}{\rule{0.2cm}{0.2cm}}

\newcommand{\re}{\mathbb{R}}

\hypersetup{ colorlinks=true, linkcolor=black, filecolor=black, urlcolor=black }
\def\proof{\noindent {\it Proof. $\, $}}

\begin{document}
\title{The generalized Duhamel principle for fully coupled systems of fractional order} 
\author[Sabir Umarov]{Sabir Umarov}
\date{\today}
\address{University of New Haven, Dept of Mathematics}
\email{sumarov@newhaven.edu}

\let\thefootnote\relax
\footnotetext{MSC2020: Primary 34A08} 

\begin{abstract}
Duhamel’s principle reduces the Cauchy problem for an inhomogeneous partial differential equation to the corresponding homogeneous problem. In the fractional-order setting, the classical principle does not apply directly because fractional derivatives are nonlocal in time. Over the past two decades, several fractional analogues of Duhamel’s principle have been developed to address this issue.

In this paper, we establish a fractional version of Duhamel’s principle for fully coupled systems of fractional differential-operator equations. The result provides a systematic reduction of inhomogeneous fractional problems to homogeneous ones while preserving the structure of the classical method. In the limit of integer-order derivatives, the formulation recovers the classical Duhamel principle and also reveals effects specific to coupled fractional systems, including those produced by coupled fractional impulses.
\end{abstract} 

\maketitle

\section{Introduction} 
\label{sec:1}

\setcounter{section}{1} \setcounter{equation}{0} 

Duhamel’s principle is a classical tool in the theory of linear evolution equations, providing a method to reduce {inhomogeneous problems} to {homogeneous ones}. Fourier’s early work on the heat equation 
laid the analytical foundations for superposition methods and integral representations of solutions. Building on these ideas,  Duhamel formulated in the 1830s a general principle for reducing the Cauchy problem for an inhomogeneous linear partial differential equation to the corresponding homogeneous problem. This reduction principle, now known as Duhamel’s principle, was subsequently employed by Kirchhoff in the study of wave propagation, leading to explicit solution formulas for the wave equation. Later, Hadamard placed these representations within a rigorous analytical framework by developing the modern theory of the Cauchy problem for hyperbolic equations, clarifying issues of well-posedness, causality, and the validity of Kirchhoff type formulas. Since then, Duhamel’s principle has become a standard tool in the analysis of linear evolution equations, with applications ranging from parabolic and hyperbolic partial differential equations to nonlinear dispersive and quantum-mechanical models \cite{Sogge,Tao}.

In recent decades, the development of {fractional evolution equations} has motivated attempts to generalize Duhamel’s principle to nonlocal-in-time operators, such as Caputo and Riemann--Liouville derivatives \cite{US06,US07,Umarov2012}.

In the classical case, for a linear evolution equation with a first-order time derivative
\begin{equation*}
\frac{\partial u}{\partial t} = A u + f(t), \quad u(0) = 0,
\end{equation*}
the solution of an inhomogeneous equation can be expressed as a {superposition of solutions of homogeneous problems}
\[
\frac{\partial v}{\partial t} = Av, \ t>\tau, \quad v(\tau) = f (\tau),
\] 
driven by \emph{instantaneous impulses} at each past time $\tau.$ Namely, 
\[
u(t) = \int_0^t e^{(t-\tau)A} f(\tau)\, d\tau,
\]
where $f(\tau)$ acts as a pointwise impulse initiating homogeneous evolution from time $\tau$. This representation relies critically on the semigroup property and the local-in-time nature of the derivative, allowing {reducibility} of the nonhomogeneous problem to homogeneous subproblems.

For fractional evolution equations with a Caputo derivative
\begin{equation*}
D_t^\alpha u(t) = A u(t) + f(t), \quad 0<\alpha<1, \quad u(0)=0,
\end{equation*}
the classical semigroup approach fails due to \emph{memory effects}: forcing at time $\tau$ affects all future times nonlocally. Consequently, the classical Duhamel reduction cannot be applied directly.
The fractional analogue of Duhamel’s principle was discussed in \cite{US06,US07}, leading to the following formula:
\begin{equation*}
u(t) = \int_0^t S_\alpha(t-\tau) D_\tau^{1-\alpha} f(\tau)\, d\tau,
\end{equation*}
where $S_\alpha(t)$ is the solution operator of the homogeneous fractional problem 
\begin{equation}
\label{fr_impulse}
D_t^\alpha v = A v, \ t>\tau,  \quad v(\tau)=D_\tau^{1-\alpha} f(\tau),
\end{equation}
with the "fractional impulse” applied to the forcing. Here, each integrand can be viewed as a homogeneous fractional problem initiated at time $\tau$ by a fractional impulse, but the evolution now depends on the entire past history rather than only the initial value at $\tau$. The key features of this approach are: 
\begin{enumerate}
\item
 the  fractional Duhamel principle reduces the inhomogeneous problem to homogeneous fractional equations;
\item the semigroup property and time-translation invariance are lost, but the essential mechanism is preserved: the same solution operator associated with the homogeneous system is used to construct the solution of the nonhomogeneous problem through a memory-integral (Duhamel-type) representation;
\item in the limit $\alpha \to 1$, the classical Duhamel formula is recovered.
\end{enumerate}

There is a substantial literature on fractional order evolution differential-operator equations. Below, we mention only works that are closely related to the representation of solutions to such equations. Kochubei \cite{Koc89} studied existence and uniqueness of a solution to the abstract Cauchy problem $D_{\ast}^{\alpha}u(t)=Au(t), \, u(0)=u_0,$ with Caputo-Djrbashian fractional derivative for $0<\alpha<1$ and a closed operator $A$ with a dense domain ${\mathcal{D}}(A)$ in a Banach space. El-Sayed \cite{E-S95} and Bazhlekova \cite{Bazhlekova2001} investigated the Cauchy problem for $0<\alpha<2.$ In the more general case of $\alpha >0,$ Gorenflo et. al. \cite{GLZ99} studied existence of solutions in Roumieu-Beurling and Gevrey classes. Kostin \cite{Kos93} proved that the abstract initial value problem (Cauchy type problem) $D_+^{\alpha}u(t)=Au(t), \, D_+^{\alpha-k}u(0)=\varphi_k, \, k=1,...,m,$ where $\alpha \in (m-1,m)$ and  $D_+^{\alpha}$ is the
Rieman-Liouville derivative, is well-posed. For more information
about recent results on the Cauchy problem for abstract fractional
differential-operator equations, we refer the reader to
\cite{Bazhlekova2001,EidKoch2004,KilbasST}; and for a recent mathematical
treatment of the distributed fractional order differential equations
to papers \cite{Kochubei2008,MSh,UG2005}.
These approaches to fractional evolution equations, including Bazhlekova \cite{Bazhlekova2001}, Mainardi et al. \cite{Mainardi2001}, and Kochubei \cite{Koc89}, provide rigorous solution representations and fundamental solutions. However, these methods do \emph{not} produce a Duhamel-type reduction: the inhomogeneous solution is not expressed as a superposition of homogeneous problems initiated by fractional impulses. 
The  fractional analog of Duhamel's principle, in the particular case of fractional order partial differential equations with a single "time-fractional" term in the evolution equation, was obtained in \cite{US06,US07}, for distributed order evolution equations in \cite{Umarov2012}. See also \cite{Umarov2015,Umarov2019} for the survey of the topic and \cite{Ali,Alloubaa,Neto,Chen,GaoWang,Ibrahim,Li,Mijena_1,Mijena_2,Mirjana1,Umarov2019,Wen} for various applications of the fractional Duhamel principle.
As regards evolution equations involving the fractional derivative in the sense of Riemann–Liouville with classical initial (Cauchy) conditions, such problems are ill-posed. Consequently, the notion of an “initial impulse,” discussed above, is not applicable. Nevertheless, one key feature of the Duhamel principle can be retained. Specifically, in an appropriate formulation, the same solution operator associated with the homogeneous problem can also be used to solve the nonhomogeneous problem.

In the current paper we establish a fractional analog of Duhamel's principle for fully-coupled distributed order systems of differential-operator equations of the following form
\begin{equation}
\label{s500}
A \circ \mathfrak{D}^{\mathscr{A}} U(t) = F \ U(t)+H(t), \quad t>0, 
\end{equation}
with the initial condition 
\begin{equation}
\label{s500i}
U(0)=\Phi.  
\end{equation}
The operator \( A \circ \mathfrak{D}^{\mathscr{A}} \) on the left-hand side of system \eqref{s500} is defined as the Hadamard (entry-wise) product of the matrix 
$A = (a_{j,k})_{j,k=1}^m$ and the matrix-valued operator 
$\mathfrak{D}^{\mathscr{A}} = (D^{\alpha_{j,k}})_{j,k=1}^m,$ 
where $D^{\alpha_{j,k}} $ denotes the Riemann--Liouville or Caputo--Djrbashian fractional derivative of order $ \alpha_{j,k}.$
That is, the operator on the left hand side  of \eqref{s500} reads:
\begin{equation}
\label{CapDer}
A \circ \mathfrak{D}^{\mathscr{A}}=
\begin{bmatrix} a_{11} {D}^{ \alpha_{11}} & a_{12} {D}^{\alpha_{12}} & \dots & a_{1m}{D}^{\alpha_{1m}} \\
a_{21}{D}^{\alpha_{2 1}} & a_{22} {D}^{\alpha_{22}} & \dots & a_{2m} {D}^{\alpha_{2m}} \\
\dots & \dots & \dots & \dots \\
a_{m1} {D}^{\alpha_{m 1} }& a_{m2} {D}^{\alpha_{m2}} & \dots &  a_{mm} {D}^{\alpha_{mm}}
\end{bmatrix},
\end{equation}
The order $\mathscr{A}$ of the operator $\mathfrak{D}^{\mathscr{A}}$ is an $m \times m$-matrix-valued order, with entries $0 < \alpha_{kj} \le 1.$
The vector-valued function $U(t)$ in \eqref{s500} with components $u_j(t), j=1,\dots,m,$ is unknown  and $\Phi$ and $H(t)$ are given vector-valued functions with components $\phi_j, j=1,\dots, m,$ and $h_j(t), j =1,\dots, m,$ respectively, with values in a Banach space $\mathcal{X};$ $F$ is a $m\times m$-matrix-valued operator with entries $f_{ij}, i,j=1,\dots,m,$ acting on  $\X.$

The precise conditions on the matrix $A,$ the matrix-valued operator $F,$ and the vector-valued functions $H(t)$ and  $\Phi$ that ensure the existence and uniqueness of a solution to problems \eqref{s500} and \eqref{s500i} will be specified in Section \ref{sec:2}, along with a review of the classical Duhamel principle and the basic reference vector spaces used in this paper. In Section \ref{sec:3} we present the main results: an abstract fractional analog of Duhamel’s principle for fully coupled systems together with several illustrative examples. Since the system is fully coupled, the associated fractional impulses are expected to be coupled as well, and their form of coupling is also addressed.

 \section{Preliminaries. Classical Duhamel's principle} \label{sec:2}

\setcounter{section}{2} \setcounter{equation}{0} 

For a function $g$ defined on $[0,\infty)$, under some integrability
conditions the {fractional integral of order $\beta$} with
{terminal points} $\tau$ and $t,$ is defined as \cite{SKM93}
\[
_{\tau}J^{\beta}g(t)=\frac{1}{\Gamma(\beta)}\int_\tau^{t}
(t-s)^{\beta -1}g(s)ds, \, t > \tau,
\]
where $\Gamma(\cdot)$ is Euler's gamma-function. Obviously, if
$\beta =n$ then $_{\tau}J^n$ is the $n$-fold integral of $f$ over
the interval $[\tau,t].$ By convention, $_{\tau}J^0 f(t)=f(t),$ i.e.
$_{\tau}J^{0}$ coincides with the identity operator. Operators $_{\tau}J^{\alpha}$ and $_{\tau}J^{\beta}$ are commutative, associative for any $\alpha, \beta \ge 0,$ and the family $\{_{\tau}J^{\beta}\}, \beta \ge 0,$ possesses the semigroup property:  
\begin{equation}
\label{semigroup}
_{\tau}J^{\alpha+\beta} = _{\tau}J^{\alpha} \ _{\tau}J^{\beta}, \quad \forall \ \alpha, \beta \ge 0.
\end{equation}

We denote by
$_\tau D_{+}^{\alpha}$ for $0 < \alpha < 1,$ the fractional
derivative of order $\alpha $ in the sense of Riemann-Liouville,
which is defined as 
$$
_\tau D_{+}^{\alpha}g(t) =
\frac{1}{\Gamma(1-\alpha)}\frac{d}{dt}
\int_{\tau}^{t}\frac{g(s)ds}{(t-s)^{\alpha}} , \, t >
\tau, 
$$
and $_\tau D_{+}^{0}g(t)=g(t), \,\, _\tau D_{+}^{1}g(t)=
g^{\prime}(t)$. The fractional derivative in the sense of Caputo-Djrbashian is defined by
$$
_\tau D_{\ast}^{\alpha}g(t) =
\frac{1}{\Gamma(1-\alpha)}
\int_{\tau}^{t}\frac{g^{\prime}(s)ds}{(t-s)^{\alpha}} , \, t >
\tau. 
$$
 There is a relationship between these two fractional derivatives \cite{GM97}. Namely,
\begin{equation}
\label{relation2} _\tau D_{+}^{\alpha}g(t) =
\,_{\tau}D_{*}^{\alpha}g(t) +
g(\tau)\frac{(t-\tau)^{-\alpha}}{\Gamma(1-\alpha)}, \quad t>\tau.
\end{equation}
If $g(\tau)=0,$ then one obtains the equality $_\tau
D_{+}^{\alpha}g(t) = \,_{\tau}D_{\ast}^{\alpha}g(t).$
 Alternative representations via the fractional integral
 are:
$$\,_{\tau}D_{+}^{\alpha}g(t)=\frac{d}{dt} \,_{\tau}J^{1-\alpha}g(t) ~~~ \mbox{and} ~~~
\,_{\tau}D_{\ast}^{\alpha}g(t)= \,_{\tau}J^{1-\alpha}\frac{d
g(t)}{dt}.$$ We omit from the notation the lower terminal point
$\tau$ if $\tau=0,$ writing simply $D_{+}^{\alpha}, \,
D_{*}^{\alpha}$ or $J^{\alpha}.$ Recall that for the Laplace
transform of $J^{\alpha}, \ D_{+}^{\alpha}g(t)$ and $D_{\ast}^{\alpha}g(t),$ 
the following formulas hold: 
\begin{equation}
\label{LI}
{\mathcal{L}}[J^{\alpha} g](s)=s^{-\alpha} \mathcal{L}[g](s), \quad s>0,
\end{equation}
\begin{equation}
\label{LRL}
{\mathcal{L}}[D_{+}^{\alpha}g](s) = s^{\alpha} {\mathcal{L}}[g](s) -
\left(J^{1-\alpha} g
\right)_{(t=0+)} s^{\alpha-1}, \quad s>0,
\end{equation}
and
\begin{equation}
\label{laplace} {\mathcal{L}}[D_{\ast}^{\alpha}g](s) = s^{\alpha}
{\mathcal{L}}[g](s) - 
g (0+) s^{\alpha -1}, \quad s>0.
\end{equation}
Here ${\mathcal{L}}[g](s)$ denotes the Laplace transform of $g.$

Now let us recall the classical Duhamel principle in the form convenient for us.
 Let \,$B=B(x,\frac{\partial}{\partial t},
 D_x),$\, where
$D_x=(\frac{\partial}{\partial x_1},...,\frac{\partial}{\partial
x_n}),$ be  a linear differential operator containing only spatial derivatives, and with coefficients not depending on
\,$t$ (see \cite{BJS64}).
Consider the Cauchy problem
\begin{equation}
\label{classic1} 
\frac{\partial u}{\partial t}(t,x) + B u(t,x) = h(t,x), \quad t > 0,\,\,x \in R^{n},
\end{equation}
with the homogeneous initial condition
\begin{equation}
\label{classic2} u(0,x) = 0. 
\end{equation}
Let a sufficiently smooth function \,$v(t, \tau, x),\,\,t \geq
\tau,\,\,\tau \geq 0,\,\, x \in R^{n},$\, be for \,$t > \tau$\, a
solution of the homogeneous equation
$$
\frac{\partial v}{\partial t}(t, \tau, x) + B v(t, \tau, x) =
0,
$$
satisfying the following conditions:
$$
v(t, \tau, x)|_{t=\tau} 
 = h(\tau, x).
$$
Then a solution of the Cauchy problem (\ref{classic1}),
(\ref{classic2}) is given by means of the integral
\begin{equation}
\label{integrel} u(t,x) = \int_{0}^{t}v(t, \tau, x)d\tau.
\end{equation}
The formulated statement is known as \textit{Duhamel's principle,}
and the integral in (\ref{integrel}) as \textit{Duhamel's integral.} It is well known that in this case the semigroup property holds, producing for the solution $u(t,x)$ the representation through the infinitesimal impulses:
\[
u(t,x) = \int_0^t e^{-(t-\tau)B} h(\tau, x) d \tau.
\]
Moreover, if the initial condition for equation \eqref{classic1} has the form $u(0)=\vf(x)$ (instead of \eqref{classic2}), then the solution has the representation
\[
u(t,x) = e^{-Bt} \vf(x) + (e^{-B \, t}) \ast h (t,x),
\]
where the symbol $"\ast"$ means the convolution operation. The importance of the latter representation is both terms use the same solution operator $S(t)=e^{-B t}.$

A similar statement (replacing the semigroup structure to the group structure) is valid in the case of the Cauchy problem with
a homogeneous initial conditions for a second order inhomogeneous
partial differential equation
$$
\frac{\partial^2 u}{\partial t^2}(t,x) + Cu(t,x) = h(t,x), \quad t > 0,
\,\,x \in R^{n},
$$
where \,$C=C(x,D_x)$\, is a linear differential operator with coefficients not depending on \,$t$, and containing temporal derivatives of order not higher than 1 
(see e.g.  \cite{BJS64}). The initial conditions for the corresponding homogeneous equation in this case have the form
$$
v(t, \tau, x)|_{t=\tau} = 0,\quad \frac{\partial v}{\partial t}(t, \tau, x)|_{t=\tau}
 = h(\tau, x).
$$

 Further, let $F$ be a linear closed operator defined on a Banach space $\mathcal{X}$ and $h(t)$ is a vector-function with values in $\mathcal{X}.$ For abstract evolution processes Duhamel's principle is formulated as follows.
\begin{theorem} (Duhamel's Principle) 
\label{DP}
The solution of the Cauchy problem
\[
\frac{du(t)}{dt} = Fu(t) + h(t), \quad u(0)=0,
\]
has the representation through Duhamel's integral
\[
u(t)=\int_0^t v(t,\tau) d\tau,
\]
where $v(t,\tau)$ is a solution to the Cauchy problem
\[
\frac{dv(t,\tau)}{dt} = F v(t,\tau), \ t>\tau, \quad v(t,\tau)\Big|_{t=\tau+}=h(\tau).
\]
\end{theorem}

\begin{remark} Theorem \ref{DP} is valid for the system of differential-operator equations as well, if one takes as $U(t)=(u_1(t),\dots,u_m(t)),$ $H(t)=(h_1(t),\dots,h_m(t)).$ Then the solution of the nonhomogeneous system $\frac{dU}{dt}=F U +H(t)$ with a matrix-valued operator $F$ and the initial condition $U(0)=0,$ has the representation
$
U(t) = \int_0^t V(t,\tau) d\tau,
$
where vector-function $V(t,\tau)$ satisfies the corresponding homogeneous system with the initial condition $V(t,\tau+)=H(\tau).$
\end{remark}

\begin{remark} Theorem \ref{DP} is generalized \cite{US06,US07} to the fractional differential-operator equations of order $0<\alpha<1:$ 
\begin{align*}
D_{\ast}^{\alpha} u(t) &= F u(t) + h(t), \ t>0, \quad 
 u(0) = 0.
\end{align*}
The solution in this case has the representation
$
u(t) = \int_0^t v(t,\tau) d\tau,
$
where vector-function $v(t,\tau)$ satisfies the initial-value problem
\begin{align*}
D_{\ast}^{\alpha} v(t,\tau) &= F v(t,\tau), \quad t>0, \quad v(t,\tau+) = D_+^{1-\alpha}h(\tau).
\end{align*}
\end{remark}

\section{The fractional Duhamel principle for fully coupled systems} \label{sec:3}

\setcounter{section}{3} \setcounter{equation}{0} 

This section formulates fractional generalizations of Duhamel’s principle and discusses representative examples. We start with preliminary lemmas.

\begin{lemma} \label{lemma1}
Let $h(t)$ be a continuous differentiable function for all $0 \le t
< T <\infty$. Then the equation $J^{\alpha}u(t)=h(t), \, 0<t<T,$
where $0<\alpha<1,$ has a unique continuous solution given by the
formula
\begin{equation} \label{Abel}
u(t)=D_+^{\alpha}h(t), \, 0<t<T.
\end{equation}
\end{lemma}

Lemma \ref{lemma1} is essentially the well-known result on a solution of Abel's integral equation of first kind. See
\cite{GM97,SKM93} for the proof.

\begin{lemma} \label{lemma2}
Suppose $v(t,\tau)$ is a $X$-valued vector-function defined for all $t \ge \tau \ge 0$ and the derivative $\frac{\partial
v(t,\tau)}{\partial t},$ is jointly continuous in the $X$-norm. Let $u(t)=\int_0^t v(t,\tau)d\tau.$ Then
\begin{equation} \label{l2}
\frac{d}{dt}u(t)=  \frac{\partial}{\partial t}v(t,\tau){\Big|_{\tau=t}}+\int_0^t \frac{\partial}{\partial t}v(t,\tau)d\tau.
\end{equation}
\end{lemma}

This lemma is a vector-function version of the well-known Leibniz integral rule for the derivative of integrals dependent on a parameter.

To formulate the next lemma, we introduce some vector-function spaces. Let $C([0,T],\mathcal{X})$ and $AC([0,T],\mathcal{X})$ denote the spaces of continuous and absolutely continuous $\mathcal{X}$-valued vector functions on $[0,T]$, respectively. The H\"older space of order $\lambda \in (0,1)$ is denoted by $C^{\lambda}([0,T],\mathcal{X})$.

\begin{lemma}
\label{volterra1}
Let $0<\alpha_{ij}<1$, $i,j=1,\dots,m$, and let $A = \{a_{ij}\}$ be a matrix with diagonal entries $a_{jj}=1$ that is invertible. Let 
\[
h_1(t),\dots,h_m(t)\in C^\lambda([0,T],\mathcal X),\quad \lambda > \max_{k,j}|\alpha_{jj}-\alpha_{kj}|.
\] 
Then the system of fractional-order Volterra-type integro-differential equations
\begin{equation}
\label{int_eq0}
g_k(t) + \sum_{\substack{j=1\\ j\neq k}}^m a_{kj} J^{\alpha_{jj}-\alpha_{kj}} g_j(t) = \sum_{\substack{j=1\\ j\neq k}}^m a_{kj} J^{\alpha_{jj}-\alpha_{kj}} h_j(t), \quad k=1,\dots,m,
\end{equation}
with initial conditions $g_j(0) = h_j(0)$, $j=1,\dots,m$, has a unique solution 
\[
(g_1(t),\dots,g_m(t)) \in C([0,T], \mathcal X)^m.
\]
\end{lemma}

\proof
Let 
$
\beta_{kj} := \alpha_{jj} - \alpha_{kj} \in (-1,1).
$
We rewrite system \eqref{int_eq0} as
\begin{equation}
\label{eq:reform}
g_k(t) = \sum_{j\neq k} a_{kj} J^{\beta_{kj}} \big( h_j(t) - g_j(t) \big), \quad k=1,\dots,m.
\end{equation}
Define the linear operator $\mathcal K: C([0,T],\mathcal X)^m \to C([0,T],\mathcal X)^m$ by
\[
(\mathcal K g)_k(t) := \sum_{j\neq k} a_{kj} J^{\beta_{kj}} g_j(t),
\]
and the forcing term
\[
h^{\ast}_k(t) := \sum_{j\neq k} a_{kj} J^{\beta_{kj}} h_j(t).
\]
Then \eqref{eq:reform} becomes
\[
G(t) = H^{\ast}(t) - \mathcal K G(t) \quad \Longleftrightarrow \quad (I + \mathcal K) G(t) = H^{\ast}(t),
\]
where $G(t)=\big(g_1(t),\dots,g_m(t) \big)$ and $H^{\ast}(t)=\big(h^{\ast}_1(t),\dots,h^{\ast}_m(t) \big).$ In system \eqref{int_eq0}
 for $\beta_{kj}>0$, the operator $J^\beta_{kj}$ is the Riemann-Liouville fractional integral,
and    for $\beta_{kj}<0$ the operator $J^{\beta_{kj}}=D_+^{\beta_{kj}}$ is the fractional derivative in the sence of Riemann-Liouville.
Since $h_j \in C^\lambda([0,T],\mathcal X)$ with $\lambda > |\beta_{kj}|$, standard fractional calculus results \cite{SKM93,Umarov2015} imply
\begin{equation}
\label{est}
J^{\beta_{kj}} h_j \in C([0,T],\mathcal X), \quad \|J^{\beta_{kj}} h_j\|_{C([0,T])} \le C \|h_j\|_{C^\lambda([0,T])},
\end{equation}

It is well known that fractional derivatives in the sense of Riemann--Liouville may produce singularities at $t=0$. In the present setting, however, no singularity arises due to the imposed initial conditions. Indeed, for $\beta_{kj}>0$ one has
$
J^{\beta_{kj}} f(0)=0,
$
and therefore no singular behavior occurs at $t=0$ in this case. 
If, on the other hand, $\beta_{kj}<0$ for some $j$ and $k$, then a singularity generally appears near $t=0$ (see, e.g., \cite{SKM93}). Applying the corresponding representation to $g_j$ and $h_j$, we obtain
\[
\begin{aligned}
D_+^{-\beta_{kj}} g_j(t)
&=
\frac{g_j(0)}{\Gamma(1-\beta_{kj})}\, t^{-\beta_{kj}}
+
\frac{1}{\Gamma(-\beta_{kj})}
\int_0^t (t-s)^{-\beta_{kj}-1}\bigl[g_j(s)-g_j(0)\bigr]\,ds,\\
D_+^{-\beta_{kj}} h_j(t)
&=
\frac{h_j(0)}{\Gamma(1-\beta_{kj})}\, t^{-\beta_{kj}}
+
\frac{1}{\Gamma(-\beta_{kj})}
\int_0^t (t-s)^{-\beta_{kj}-1}\bigl[h_j(s)-h_j(0)\bigr]\,ds .
\end{aligned}
\]
Since the initial conditions satisfy $g_j(0)=h_j(0)$ for all $j=1,\dots,m$,
the singular terms of order $t^{-\beta_{kj}}$ cancel algebraically in
\eqref{int_eq0}.
As a consequence, only the regular terms remain, and the resulting expressions are continuous in a neighborhood of $t=0$.

The operator $\mathcal K$ is obviously of Volterra type if all $\beta_{kj}>0$.
The Volterra structure remains valid also in the presence of terms with $\beta_{kj}<0$. 
Indeed,  after cancellation of singular terms,   for $h\in C^\lambda([0,T])$ with $\lambda>-\beta_{kj},$ the operator
$J^{\beta_{kj}}$  acts only through the integral term
\[
\frac{1}{\Gamma(-\beta_{kj})}
\int_0^t (t-s)^{-\beta_{kj}-1}\bigl[g_j(s)-g_j(0)\bigr]\,ds,
\]
which is an ordinary Volterra integral operator. Since
$g_j(s)-g_j(0)=O(s^\lambda)$ as $s\to0$ and $\lambda>-\beta_{kj}$, the kernel
is locally integrable and the resulting mapping is continuous on
$[0,T]$. Therefore, the operator $\mathcal K$ is of Volterra type also in the presence of terms with $\beta_{kj}<0$.

 Further, standard estimates yield
\begin{equation}
\label{est1}
\|\mathcal K^n\| \le \frac{C_0^n T^n}{n!}, \quad n \ge 1,
\end{equation}
for some constant $C_0>0$ depending on $\|a_{kj}\|$, $\beta_{kj}$, $\lambda$, and $T$.  
This estimate implies that for the spectral radius of the operator $\mathcal{K},$ we have
\[
r(\mathcal K) = \lim_{n\to\infty} \|\mathcal K^n\|^{1/n} = 0,
\]
so $\mathcal K$ is quasi-nilpotent. It follows that the operator $I+\mathcal K$ is invertible with bounded inverse
\[
(I+\mathcal K)^{-1} = \sum_{n=0}^\infty (-\mathcal K)^n,
\]
which converges in the operator norm. Hence, system \eqref{int_eq0} has a unique solution
$
G(t) = (I+\mathcal K)^{-1} H^{\ast}(t).
$
Moreover, utilizing estimate \eqref{est1}, one obtains
\[
\|(I+\mathcal K)^{-1}\|
\le
\sum_{n=0}^\infty \|\mathcal K^n\|
\le
\sum_{n=0}^\infty \frac{(C_0T)^n}{n!}
=
e^{C_0T}.
\]
Now using estimate \eqref{est}, we have
\[
\|H^{\ast}\|_{C([0,T])^m}
\le
C_1\,\|H\|_{C^\lambda([0,T])^m}, \quad H(t)=\big(h_1(t),\dots,h_m(t) \big).
\]
Therefore, the solution satisfies the estimate
\begin{align*}
\|G\|_{C([0,T])^m} &\le \| (I+\mathcal K)^{-1} H^{\ast}(t) \|_{C([0,T])^m} 
\\
&\le \|(I+\mathcal K)^{-1}\| \|H^{\ast}\|_{C([0,T])^m}
\le
C\,\|H\|_{C^\lambda([0,T])^m},
\end{align*}
where the constant $C=C_1e^{C_0T}>0$ depends only on $T$, the coefficients $a_{kj}$, and the
fractional orders $\alpha_{ij}$.
\eproof

\begin{remark}
\begin{enumerate}
\item If $h_j \in AC([0,T],\mathcal X)$ and additionally $h_j(0)=0$, then the above argument still works, and
 the system \eqref{int_eq0} has a unique solution $g \in AC([0,T],\mathcal X)^m.$ 
\item
If $a_{kj}=0$ for $k\neq j,$ then the right hand side of system \eqref{int_eq0} is zero. The system in this case reduces to
\[
J^{1-\alpha_{jj}}g_j(t) =0, \quad j=1,\dots, m,
\]
implying $g_j(t)=0, j=1,\dots,m.$
\item If $\alpha_{kj} =1$ for all $k,j=1,\dots, m,$ then the solution of system \eqref{int_eq0} is represented as $G(t)=A^{-1} \mathcal{H}(t),$ where the vector-function $\mathcal{H}(t)$ has components
\[
\mathcal{H}_k=  \sum_{\substack{j=1\\j \neq k}}^{m} a_{kj} h_j(t), \quad k=1,\dots,m.
\]
\end{enumerate}
\end{remark}

We now proceed to formulate the main results. Consider the system of fully coupled fractional order evolution equations
\begin{equation}
\label{eq20}
A\circ \mathfrak{D}^{\mathcal{A}} U(t) = F U(t) + H(t), \quad t>0,  
\end{equation}
with the initial condition
\begin{equation}
\label{cauchy20}
 U(0)= \Phi = (\vf_1,\dots,\vf_m).
\end{equation}
where $H(t)=\big( h_1(t), \dots, h_m(t) \big) \in C^{\lambda}([0,T], \X)^m.$
We assume that the diagonal entries of the matrix $A = \{a_{ij}\}$ are not zero, that is $a_{jj} \neq 0, j=1,\dots,m.$ If necessary dividing $j$-th equation in the system by $a_{jj},$ we can assume that $a_{jj}=1$ for $j=1,\dots,m.$

The fractional Duhamel principle establishes a connection between the solutions of problem \eqref{eq2}-\eqref{cauchy2} and the Cauchy problem for the
homogeneous equation
\begin{equation}
\label{eq50}
A\circ \mathfrak{D}^{\mathcal{A}} V(t,\tau) = F V(t,\tau), \quad t>\tau,  
\end{equation}
with the nonhomogeneous initial condition at any time instant $t=\tau:$
\begin{equation}
\label{cauchy510}
 V(t, \tau) \Big|_{t=\tau+}=  I \circ \mathfrak{D}_{+,\tau}^{I-\Lambda} \Big(H(\tau) - G(\tau)\Big),
\end{equation}
where $I$ is the identity matrix of size $m,$ $\Lambda=I \circ \mathcal{A},$ and $G(t)=(g_1(t),\dots, g_m(t))$ 
is the solution to the system \eqref{int_eq0}.

\begin{remark} 
In the case of a single fractional evolution equation, the instantaneous fractional impulses $D_+^{1-\alpha}h(\tau)$ in the initial conditions of the homogeneous equation (see Eq. \eqref{fr_impulse}) carry the memory of the past. For fully coupled fractional dynamical systems, however, as seen from \eqref{cauchy51}, the instantaneous impulses depend not only on the past of each individual component but also on the entire history of all coupled quantities.
\end{remark}

\begin{theorem} \label{thm1} Suppose that \,$V(t,\tau)$ 
is a solution of the Cauchy problem \eqref{eq50},\eqref{cauchy510}.
Then Duhamel's integral
\begin{equation}
\label{solution2} 
U(t) = \int_{0}^{t}V(t,\tau)d\tau, 
\end{equation}
solves Cauchy problem \eqref{eq20}, \eqref{cauchy20} with  $\Phi=0$.
\end{theorem}

\proof 
The component-wise  form of Cauchy problem \eqref{eq20}, \eqref{cauchy20} with the homogeneous initial condition is written as
\begin{equation}
\label{eq2}
 \sum_{j=1}^m a_{ij} D_{\ast}^{\alpha_{ij}}u_{j}(t) = \sum_{j=1}^{m}f_{ij} u_{j}(t) + h_i(t), \quad t>0, \ i=1,\dots,m,
\end{equation}

\begin{equation}
\label{cauchy2}
 u_{i}(0)= 0, \quad i = 1, \dots , m,
\end{equation}
while the component-wise form of  the Cauchy problem \eqref{eq50},\eqref{cauchy510} as

\begin{equation}
\label{eq5}
 \sum_{j=1}^m a_{ij} \,_{\tau}D_{\ast,t}^{\alpha_{ij}}v_{j}(t,\tau) = \sum_{j=1}^{m}f_{ij} v_{j}(t,\tau), \quad t>\tau, \ i=1,\dots,m,
\end{equation}

\begin{equation}
\label{cauchy51}
 v_{i}(t, \tau) \Big|_{t=\tau+}=  D_{+,\tau}^{1-\alpha_{ii}} \Big(h_i(\tau) - g_i(\tau)\Big),
 \quad i = 1, \dots , m.
\end{equation}
Let $v_j(t,\tau), j=1,\dots,m,$ as functions of the variable $t,$ be the solution to Cauchy problem (\ref{eq5})-(\ref{cauchy51}) for any fixed $\tau.$ We
verify that the functions $u_j(t)=\int_{0}^t v_j(t,\tau)d\tau$\, satisfy equation
(\ref{eq2}), and conditions \eqref{cauchy2}. Indeed, by the lemma \ref{lemma2},  for $i=1,\dots,m,$ we have
\begin{align}
D u_j(t) &= \frac{du_j(t)}{dt} = v_j(t,\tau)\Big|_{\tau=t} + \int_0^t \frac{\partial v_j(t,\tau)}{\partial t} d\tau 
\notag
\\
&= D_{+,t}^{1-\alpha_{jj}} h_j(t) - D_+^{1-\alpha_{jj}}g_j(t)
+ \int_0^t \frac{\partial v_j(t,\tau)}{\partial t} d\tau.
\label{r1}
\end{align}
Applying to both sides of \eqref{r1} the fractional integration operator $J^{1-\alpha_{jj}}$ and using Lemma \ref{lemma1}, we obtain
\[ 
D_{\ast}^{\alpha_{ii}}u_j(t) = h_j(t) -  g_j(t)
+ J^{1-\alpha_{ii}} \int_0^t  \frac{\partial v_j(t,\tau)}{\partial t} d\tau
\]
Now, utilizing the semigroup property \eqref{semigroup} of fractional integration operators in the last term of the right of the latter, we have
\begin{equation}
\label{r2}
D_{\ast}^{\alpha_{ii}}u_i(t) = h_i(t) -  g_i(t)
+  \int_0^t \Big( D_{\ast,t}^{\alpha_{ii}} v_i(t,\tau) \Big)  d\tau
\end{equation}
Here, the subscript $t$ in the fractional derivative $D_{\ast,t}^{\alpha_{ii}} $ indicates that the operator  acts with respect to the variable $t.$
Similarly, applying to both sides of \eqref{r1} the operator $a_{ij} \ J^{1-\alpha_{ij}}, \ j \neq i$, and taking into account the equality
\[
J^{1-\alpha_{ij}}D_+^{1-\alpha_{jj}} h_j(t)= J^{\alpha_{jj} - \alpha_{ij}}  J^{1-\alpha_{jj}} D_+^{1-\alpha_{jj}} h_j(t) = J^{\alpha_{jj}-\alpha_{ij}} h_i(t),
\]
we obtain
\begin{align}
a_{ij} D_{\ast}^{\alpha_{ij}}  u_j(t) &= {a_{ij}} J^{\alpha_{jj}-\alpha_{ij}} 
h_j(t) - a_{ij} J^{\alpha_{jj}-\alpha_{ij}}  g_j(t)
\notag
\\
&+   \int_0^t \Big(a_{ij} D_{\ast,t}^{\alpha_{ij}} v_j(t,\tau) \Big) d\tau. \label{r3}
\end{align}
Note that if $\alpha_{jj}< \alpha_{ij}$ for some $i$ and $j,$ then the operator $J^{\alpha_{jj}-\alpha_{ij}}$ in the first and second terms on the right hand side of \eqref{r3} becomes $D_{+}^{\alpha_{ij}-\alpha_{jj}}.$ The possible singularities that emerge near $t=0$ in this case cancel due to the condition $g_j(0)=h_j(0)$ required for $g_j(t)$ in system \eqref{int_eq0}. Equations \eqref{r2} and \eqref{r3} imply that
\begin{align}
\sum_{j=1}^m a_{ij} D_{\ast}^{\alpha_{ij}} u_j(t) &= h_i(t) + \sum_{\substack{j=1 \\ j \neq i}}^{m} {a_{ij}}   J^{\alpha_{jj}-\alpha_{ij}}h_j(t) - \sum_{j=1}^m a_{ij} J^{\alpha_{jj}-\alpha_{ij}}  g_j(t)
\notag
\\
&+ \int_0^t \left( \sum_{j=1}^m a_{ij} D_{\ast,t}^{\alpha_{ij}} v_j(t,\tau) \right) d\tau.
\label{r4}
\end{align}
Since $g_j(t), j=1,\dots,m,$ satisfy \eqref{int_eq0}, the second and third terms on the right hand side of \eqref{r4} cancel. Moreover, due to equation \eqref{eq5} the last term on the right hand side of \eqref{r4} takes the form
\begin{align*}
 \int_0^t \left( \sum_{j=1}^m a_{ij} D_{\ast,t}^{\alpha_{ij}} v_j(t,\tau) \right) d\tau &= \int_0^t \left( \sum_{j=1}^m f_{ij} v_j(t,\tau) \right) d\tau 
 \\
 &= \sum_{j=1}^m f_{ij} \int_0^t v_j(t,\tau)d\tau 
 \\
 &= \sum_{j=1}^m f_{ij} u_j(t).
 \end{align*}
Consequently, the functions $u_j(t), j=1,\dots,m,$ defined by the Duhamel integrals \eqref{solution2} satisfy system \eqref{eq2}.
The fact that these functions also satisfy initial conditions \eqref{cauchy2} immediately follows from the definition of $U(t)$ by \eqref{solution2}.
\eproof

If the vector-valued function \,$H(t)$\, in system \eqref{eq20} satisfies the additional condition
\,$H(0) = 0$\, then condition (\ref{cauchy510}), in accordance with
the relationship (\ref{relation2}), can be replaced by
\begin{equation}
\label{Cc}
\frac{\partial V}{\partial t}(t,\tau)|_{t=\tau} =
I \circ \mathfrak{D}_{*}^{I-I\circ \mathcal{A}} \Big( H(\tau)-G(t) \Big).
\end{equation}
As a consequence the formulation of the fractional
Duhamel principle takes the form:

\begin{theorem} \label{thm11} Let the vector-valued function $H(t)\in \AC^m$ in system \eqref{eq20} satisfy the condition $H(0)=0.$ Suppose that \,$V(t,\tau)$ 
is a solution of the Cauchy problem \eqref{eq50}-\eqref{Cc}.
Then Duhamel's integral
\[
U(t) = \int_{0}^{t}V(t,\tau)d\tau,
\]
solves the Cauchy problem \eqref{eq20}, \eqref{cauchy20} with $\Phi=0$.
\end{theorem}

Similarly, we can formulate the fractional Duhamels principle for fully coupled systems given through the Riemann-Liuoville fractional derivatives.
Consider the system represented in the component-wise form

\begin{equation}
\label{eq9}
 \sum_{j=1}^m a_{ij} D_{+}^{\alpha_{ij}}u_{j}(t) = \sum_{j=1}^{m}f_{ij} u_{j}(t) + h_i(t), \quad t>0, \ i=1,\dots,m,
\end{equation}
\begin{equation}
\label{cauchy9}
 \Big(J^{1-\alpha_{ii}}u_{i} \Big)(0)= 0, \quad i = 1, \dots , m,
\end{equation}
where $D_+^{\alpha}$ is the fractional derivative in the sense of Riemann-Liouville, and $H(t) \in C([0,T]; \X).$
Again, without loss of generality we can assume 
that $a_{jj}=1$ for $j=1,\dots,m.$

In the Riemann-Liouville case the fractional Duhamel principle establishes a connection between the solutions of  problem \eqref{eq9}-\eqref{cauchy9} and the Cauchy problem for the
homogeneous equation
\begin{equation}
\label{eq91}
 \sum_{j=1}^m a_{ij} \,_{\tau}D_{+,t}^{\alpha_{ij}}v_{j}(t,\tau) = \sum_{j=1}^{m}f_{ij} v_{j}(t,\tau), \quad t>\tau, \ i=1,\dots,m,
\end{equation}
with non-homogeneous initial conditions at any time instant $t=\tau:$
\begin{equation}
\label{cauchy91}
 \Big( J^{1-\alpha_{ii}} _{i}  v_i \Big) (t, \tau) \Big|_{t=\tau+}=  h_i(\tau) - g_i(\tau),
 \quad i = 1, \dots , m,
\end{equation}
where $q_i(t), i=1,\dots,m,$ is the solution to the system \eqref{int_eq0}.

\begin{theorem} \label{RL} Suppose that \,$V(t,\tau)$ 
is a solution of the Cauchy problem \eqref{eq91}-\eqref{cauchy91}.
Then Duhamel's integral
\begin{equation}
\label{solution91} 
U(t) = \int_{0}^{t}V(t,\tau)d\tau, 
\end{equation}
solves the Cauchy problem \eqref{eq9}, \eqref{cauchy9}.
\end{theorem}

\proof
Let $v_j(t,\tau), j=1,\dots,m,$ as functions of the variable $t,$ be the solution to
Cauchy problem (\ref{eq91})-(\ref{cauchy91}) for any fixed $\tau.$ We
verify that the functions 
\begin{equation}
\label{v}
u_j(t)=\int_{0}^t v_j(t,\tau)d\tau
\end{equation}
satisfy equation (\ref{eq9}), and conditions \eqref{cauchy9}. Indeed, applying the fractional integration operator $J^{1-\alpha_{jj}}$ 
to \eqref{v} and using the semigroup property of fractional integration operators, we have
\begin{equation*}
J^{1-\alpha_{jj}} u_j(t)=\int_{0}^t  J_t^{1-\alpha_{ii}}v_j(t,\tau)d\tau
\end{equation*}
Now differentiating both sides of the latter and using  Lemma \ref{lemma2} and the definition of the fractional derivative in the sense of Riemann-Liouville,  for $i=1,\dots,m,$ we have
\begin{align}
D_+^{\alpha_{jj}} u_j(t) &= \Big(J^{1-\alpha_{jj}} v_j \Big) (t,\tau)\Big|_{\tau=t} + \int_0^t \frac{\partial}{\partial t} J_t^{1-\alpha_{jj}}v_j(t,\tau) d\tau 
\notag
\\
&= h_j(t) - g_j(t)
+ \int_0^t D_{+,t}^{\alpha_{jj}} v_j(t,\tau)d\tau.
\label{r11}
\end{align}
The subscript $t$ in the fractional derivative $D_{+,t}^{\alpha_{ii}} $ in the integrand of \eqref{r11} indicates that the operator  acts with respect to the variable $t.$
Similarly, applying to both sides of \eqref{v} the operator $a_{ij} \ J^{1-\alpha_{ij}}, \ j \neq i$, and differentiating, we obtain
\begin{align}
a_{ij} D_{+}^{\alpha_{ij}}  u_j(t) &= {a_{ij}} J^{\alpha_{jj}-\alpha_{ij}} 
h_j(t) - a_{ij} J^{\alpha_{jj}-\alpha_{ij}}  g_j(t)
\notag
\\
&+   \int_0^t \Big(a_{ij} D_{+,t}^{\alpha_{ij}} v_j(t,\tau) \Big) d\tau. \label{v3}
\end{align}
Here we used the equality $J^{1-\alpha_{ij}} = J^{\alpha_{jj}-\alpha_{ij}}J^{1-\alpha_{jj}},$ valid due to the semigroup property \eqref{semigroup} of the family of fractional integration operators. 
Equations \eqref{r11} and \eqref{v3} imply that
\begin{align}
\sum_{j=1}^m a_{ij} D_{+}^{\alpha_{ij}} u_j(t) &= h_i(t) + \sum_{\substack{j=1 \\ j \neq i}}^{m} {a_{ij}}   J^{\alpha_{jj}-\alpha_{ij}}h_j(t) - \sum_{j=1}^m a_{ij} J^{\alpha_{jj}-\alpha_{ij}}  g_j(t)
\notag
\\
&+ \int_0^t \left( \sum_{j=1}^m a_{ij} D_{+,t}^{\alpha_{ij}} v_j(t,\tau) \right) d\tau.
\label{v4}
\end{align}
Since $g_j(t), j=1,\dots,m,$ satisfy \eqref{int_eq0}, the second and third terms on the right hand side of \eqref{v4} cancel. Moreover, due to equation \eqref{eq9} the last term on the right of \eqref{v4} takes the form
\begin{align*}
 \int_0^t \left( \sum_{j=1}^m a_{ij} D_{+,t}^{\alpha_{ij}} v_j(t,\tau) \right) d\tau &= \int_0^t \left( \sum_{j=1}^m f_{ij} v_j(t,\tau) \right) d\tau 
 = \sum_{j=1}^m f_{ij} u_j(t).
 \end{align*}
Thus, the functions $u_j(t), j=1,\dots,m,$ defined by the Duhamel integrals \eqref{solution91} satisfy system \eqref{eq9}.
The fact that these functions also satisfy initial conditions \eqref{cauchy9} immediately follows from the definition of $u_j(t), j=1,\dots,m,$ by \eqref{solution91}.
\eproof

If the system coefficients satisfy $a_{ij}=0$ for $i \neq j,$  the system becomes decoupled in memory, and the history of each component evolves independently of the others. In this case, $Q(t) \equiv 0$  and the fractional impulses depend only on the history of the corresponding components. Consequently, we obtain the following theorem for fractional vector-order systems. 
\begin{theorem} 
\label{cor}
The solution of the Cauchy problem for the system of fractional-order evolution equations
\[
  D_{\ast}^{\alpha_{i}}u_{i}(t) = \sum_{j=1}^{m}f_{ij} u_{j}(t) + h_i(t), \ t>0, \quad u_{i}(0)= 0, i=1,\dots,m,
\]
has the representation $u_i(t)=\int_0^t v_i(t,\tau)d\tau,$ where $\big(v_1(t,\tau),\dots,v_m(t,\tau)\big)$ is a solution to the Cauchy problem
\[
  D_{\ast}^{\alpha_{i}}v_{i}(t,\tau) = \sum_{j=1}^{m}f_{ij} v_{j}(t,\tau), \ t>\tau, \quad v_i(\tau+,\tau)=D_+^{1-\alpha_{i}}h_i(\tau), \ i=1,\dots,m.
\]
\end{theorem}
\begin{remark} Conclusion similar to Theorem \ref{cor} is valid in the case of systems given by the fractional derivatives in the sense of Riemann-Liouville. In the special case $m=1,$ Theorems \ref{thm1} and \ref{thm11} recover the fractional Duhamel principle established in \cite{US06,US07}.
\end{remark}

Now consider examples to illustrate obtained results.

\begin{example}
Consider the 2$\times$2 fully coupled system of evolution equations with Caputo fractional derivatives
\begin{equation}\label{eq7}
\begin{cases}
D_{\ast}^{\alpha_{11}} u_1(t) + a_{12} D_{\ast}^{\alpha_{12}} u_2(t) = f_{11} u_1(t) + f_{12} u_2(t) + h_1(t), \\
a_{21} D_{\ast}^{\alpha_{21}} u_1(t) + D_{\ast}^{\alpha_{22}} u_2(t) = f_{21} u_1(t) + f_{22} u_2(t) +h_2(t),
\end{cases}
\end{equation}
with initial conditions
\[
u_1(0) = \vf_1, \quad u_2(0) = \vf_2,
\]
where $a_{12}a_{21} \neq 1$ and all orders $\alpha_{ij} > 0$ are different and satisfy the sum condition
\[
\alpha_{11} + \alpha_{22} = \alpha_{12} + \alpha_{21}.
\]
We also assume that $\vf_1, \vf_2 \in \X$ and $h_1, h_2 \in C^{\lambda}([0,T],\X).$

To apply Duhamel's principle, we first find the solution of the following system of Volterra type fractional order integro-differential equations

\[
\begin{cases}
 g_1(t) + a_{12} J^{\beta} g_2(t) = a_{12} J^{\beta} h_2(t), \\[4pt]
 a_{21} D_+^{\beta} g_1(t) + g_2(t) = a_{21} D_+^{\beta} h_1(t),
\end{cases}
\]
with the initial condition $q_1(0)=h_1(0)$ and where $0<\beta= \alpha_{11}-\alpha_{12} = \alpha_{21}-\alpha_{22} <\lambda.$ 
Applying the Laplace transform we obtain the algebraic system
\[
\begin{cases}
 \tilde g_1(s) + a_{12}s^{-\beta}\tilde g_2(s)
 =
 a_{12}s^{-\beta}\tilde h_2(s), \\[6pt]
 a_{21}s^{\beta}\tilde g_1(s) + \tilde g_2(s)
 =
 a_{21}s^{\beta}\tilde h_1(s).
\end{cases}
\]
Since the determinant equals $1-a_{12}a_{21}\neq 0$, the system has a unique solution
\[
\tilde g_1(s)
=
\frac{a_{12}}{1-a_{12}a_{21}}
\bigl(
s^{-\beta}\tilde h_2(s)-a_{21}\tilde h_1(s)
\bigr),
\]
\[
\tilde g_2(s)
=
\frac{a_{21}}{1-a_{12}a_{21}}
\bigl(
s^{\beta}\tilde h_1(s)-a_{12}\tilde h_2(s)
\bigr).
\]
Taking inverse Laplace transforms, we have 

\[
g_1(t)
=
\frac{a_{12}}{1-a_{12}a_{21}}
\bigl(
J^{\beta} h_2(t) - a_{21} h_1(t)
\bigr),
\]

\[
g_2(t)
=
\frac{a_{21}}{1-a_{12}a_{21}}
\bigl(
D_+^{\beta} h_1(t) - a_{12} h_2(t)
\bigr),
\]
since $(J^{\beta} h_1)(0)=0.$
Now, let
\begin{align*}
h_1^{\ast}(t) &= D_+^{1-\alpha_{11}} \Big(h_1(t) -g_1(t)\Big) \quad \text{and}
\quad
h_2^{\ast}(t) = D_+^{1-\alpha_{22}} \Big(h_2(t) - g_2(t)\Big).
\end{align*}
Then in accordance with Theorem \ref{thm1} the solution of problem \eqref{eq7} can be represented through Duhamel's integral
\[
U(t)= S(t) \Phi + \int_0^t S(t-\tau)  H^{\ast}(\tau)d\tau, 
\]
where $S(t)$ is a matrix-valued solution operator of the homogeneous system associated to \eqref{eq7}, and 
\[
\Phi=\begin{bmatrix} \vf_1 \\ \vf_2  \end{bmatrix}, \quad H^{\ast}(t) = \begin{bmatrix} h_1^{\ast}(t) \\ h_2^{\ast}(t)  \end{bmatrix}.
\] 

\end{example}

\begin{example} 
Let the matrix-valued differential operator $F(D_x)$ with  $ D_x=\frac{d}{idx},$ $ x \in \re,$ be given by the matrix-symbol 
\[
\mathcal{F}(\xi) = 
\begin{bmatrix} 
-|\xi|^2 & 0 \\
i \xi     & -|\xi|^2
\end{bmatrix}, \quad \xi \neq 0.
\]
Consider the nonhomogeneous system with matrix-valued operator on the right corresponding to the symbol $\mathcal{F}(\xi):$ 
\begin{equation}
\label{ex1}
\begin{bmatrix}
D_{\ast}^{1/2} \\ D_{\ast}^{1/3}
\end{bmatrix}
\mathcal{U}(t,x)=
\begin{bmatrix} 
D_x^2 & 0 \\
iD_x     & D_x^2
\end{bmatrix} 
\mathcal{U}(t,x) +\mathcal{H}(t,x), \quad t>0, \ x \in \re,
\end{equation}
with the initial condition 
\begin{equation}
\label{ex2}
\mathcal{U}(0,x)= \Phi(x) = (\vf_1(x),\vf_2(x)), \quad x \in \re,
\end{equation}
where $H \in AC^([0,T], L_1(\re))^2$ and $\Phi(x) \in L_1(\re)\times L_1(\re).$
In equation \eqref{ex1} the time-fractional orders are $\alpha_{11}=1/2, \alpha_{22}=1/3$ and coefficients $a_{12}=a_{21}=0,$ implying $g_1(t)=g_2(t)=0.$ Hence, in accordance with Theorem \ref{cor} the solution of the Cauchy problem \eqref{ex1}-\eqref{ex2} has
the representation
\[
U(t,x)= S(t, D_x) \Phi (x) + \int_0^t S(t-\tau,D_x) \begin{bmatrix} D_{\ast}^{1/2} \\ D_{\ast}^{2/3} \end{bmatrix} H(\tau,x)d\tau, 
\]
where $S(t,D)$ is a matrix-valued solution operator of the homogeneous system associated to \eqref{ex1}. 
The solution operator $S(t,D)$ has the matrix-symbol  \cite{Umarov2024}
\begin{align*}
S(t,\xi) 
&=\begin{bmatrix}
E_{1/2}(-|\xi|^2 t^{1/2}) & 0 \\
 i\xi \Big( E_{1/2}(-|\xi|^2 t^{1/2})\ast t^{-2/3}E_{1/3,1/3} (-|\xi|^2 t^{1/3}) \Big) & E_{1/3}(-|\xi|^2 t^{1/3})
  \end{bmatrix},
\end{align*}
where $E_{\alpha}(z)$ and $E_{\alpha,\beta}(z)$ denote the one- and two-parameter Mittag-Leffler functions, respectively.

Thus, the unique solution of Cauchy problem \eqref{ex1}-\eqref{ex2} has the representation

\begin{align} \label{u1}
u_1(t,x) &= E_{1/2}(D_x^2 t^{1/2}) \vf_1(x) + \int_0^t  E_{1/2}(D_x^2 (t-\tau)^{1/2}) D_+^{1/2}h_1(\tau,x) d\tau,      \\
u_2(t,x) &=  D_x \Big( E_{1/2}(D^2 t^{1/2})\ast t^{-2/3}E_{1/3,1/3} (D_x^2 t^{1/3}) \Big) \vf_1(x) \notag
\\
&+ \int_0^t D_x \Big( E_{1/2}(D^2 (\cdot)^{1/2}) \ast (\cdot)^{-2/3}E_{1/3,1/3} (D_x^2 (\cdot)^{1/3}) \Big)(t-\tau) D_+^{1/2} h_1(\tau,x) d \tau \notag
\\ \label{u2}
&+
E_{1/3}(D_x^2 t^{1/3}) \vf_2(x) + \int_0^t E_{1/3}(D_x^2 (t-\tau)^{1/3}) D_+^{2/3}h_2(\tau,x)d\tau.
\end{align}
Note that the operator $E_{1/2}(D_x^2 t^{1/2}) $ in \eqref{u1} is defined by (see \cite{Umarov2024})
\[
E_{1/2}(D_x^2 t^{1/2}) \vf(x) = \frac{1}{2 \pi} \int_{\re} e^{-ix\xi} E_{1/2}(-|\xi|^2 t^{1/2}) F[\vf](\xi) d\xi,
\]
where $F[\vf](\xi)$ stands for the Fourier transform of $\vf(x).$ Due to decreasing behavior of the Mittag-Leffler function $E_{1/2}(-u)$ when $u \to \infty$ \cite{Umarov2015}, the integral on the right is meaningful. The operators in \eqref{u2} are defined similarly.

\end{example}

\begin{remark}
Finding an explicit representation of the solution operator $S(t)$ for fully coupled systems is a challenging problem. We therefore devote a separate paper to the construction of solution operators for such systems.
\end{remark}

\section{Conclusions} \label{sec:6}

In conclusion, the fractional Duhamel framework highlights the fundamentally nonlocal nature of memory in fractional order dynamics. For a single fractional evolution equation, the instantaneous fractional impulse $D_+^{1-\alpha}h(\tau)$ appearing in the initial conditions of the corresponding homogeneous problem encodes the accumulated influence of the past and governs the subsequent evolution. In contrast, for fully coupled fractional dynamical systems, memory effects acquire a genuinely collective character. As indicated by \eqref{cauchy51}, the instantaneous impulses are no longer determined solely by the history of each state variable in isolation, but rather by the intertwined histories of all coupled components. This emphasizes that coupling in fractional systems does more than link present states - it entangles their pasts. However, if the system coefficients satisfy $ a_{ij} = 0 $ for all  $i \neq j$, the components remain dynamically coupled, but the fractional memory contributions of each component become independent ($Q(t) \equiv 0$). In this case, while the present dynamics of each component still evolve according to its own equation, the historical influence of one component no longer affects the others - a phenomenon referred to as memory decoupling. This distinction highlights that in fractional systems, coupling can entangle not only present states but also past histories. The fractional Duhamel principle thus provides a rigorous framework to describe how distributed memory and coupling jointly shape the evolution of complex fractional dynamical systems.

\end{document}